\documentclass[letter,11pt]{article}

\usepackage{amssymb}

\usepackage{amsthm}
\usepackage{amsfonts}
\usepackage{amsmath}
\usepackage{anysize}
\usepackage{hyperref}
\usepackage{bbm}
\marginsize{2.6cm}{2.6cm}{1cm}{2cm}

\newtheorem*{theorem*}{Theorem}

\newtheorem{theorem}{Theorem}[section]

\newtheorem{lemma}[theorem]{Lemma}
\newtheorem{corollary}[theorem]{Corollary}
\newtheorem{remark}[theorem]{Remark}

\newtheorem{proposition}[theorem]{Proposition}

\def\K{K\"ahler }

\def\del{\partial}
\def\dbar{\bar\partial}
\def\ddbar{\del\dbar}

\def\del{\partial}

\DeclareMathOperator{\Ric}{Ric}

\def\o{\omega}

\title{\vspace{-0.50in}Comparison of the Calabi and Mabuchi geometries  and applications to geometric flows\vspace{-0.1in}}
\author{Tam\'as Darvas}
\date{\vspace{-0.25in}}

\begin{document}
\maketitle
\begin{abstract} Suppose $(X,\o)$ is a compact K\"ahler manifold. We introduce and explore the metric geometry of the $L^{p,q}$-Calabi Finsler structure on the space of K\"ahler metrics $\mathcal H$. After noticing that the $L^{p,q}$-Calabi and $L^{p'}$-Mabuchi path length topologies on $\mathcal H$ do not typically dominate each other, we focus on the finite entropy space $\mathcal E^\textup{Ent}$, contained in the intersection of the $L^p$-Calabi and $L^1$-Mabuchi completions of $\mathcal H$ and find that after a natural strengthening, the $L^p$-Calabi and $L^1$-Mabuchi topologies coincide on $\mathcal E^\textup{Ent}$. As applications to our results, we give new convergence results for  the K\"ahler--Ricci flow and the weak Calabi flow.
\end{abstract}

\section{Introduction and Main Results}\vspace{-0.1cm}

Suppose $(X^n,\o)$ is a connected compact K\"ahler manifold. By $\mathcal H$ we denote the space of K\"ahler metrics $\o'$ that are cohomologous to $\o$. In the 1950's Calabi initiated the study of the infinite-dimensional space $\mathcal H$, with the hopes of finding K\"ahler metrics with special curvature properties \cite{clb1}. He introduced a Riemannian structure on $\mathcal H$ and formulated many related questions, including his famous conjecture ultimately solved by Yau \cite{yau}. Addressing one of Calabi's predictions, the path length completion of Calabi's Riemannian space was computed by Clarke--Rubinstein \cite{cr}, and they also found a novel relation between Calabi geometry and convergence of the K\"ahler--Ricci (KR) flow on Fano manifolds. One of our purposes in the present paper is to further develop this circle of ideas. The KR flow was introduced by Hamilton \cite{ham}, it satisfies the equation $\frac{d}{dt} \o_{r_t}= \o_{r_t}-\Ric \o_{r_t}$ with $\o \in c_1(X)$, and we refer to \cite{beg} for background, context and historic references. 

In our first theorem, refining the findings of Clarke-Rubinstein $(p=2,q=1)$ \cite{cr} and also that of of Phong--Song--Sturm--Weinkove $(p=\infty,q=1)$ \cite{pssw}, we obtain the following convergence theorem for the KR flow, whose statement on the surface bears no connection with infinite-dimensional geometry:
\begin{theorem}\label{thm: KRconvergence}Suppose $1 \leq q \leq p \leq \infty, \ q \neq \infty$ and $(X^n,\o)$ is Fano with $\o \in c_1(X)$. Let $[0,\infty) \ni t \to \o_{r_t} \in \mathcal H$ be a KR trajectory. Then a K\"ahler--Einstein metric in $\mathcal H$ exists if and only if \vspace{-0.1cm}
\begin{equation}\label{eq: KRCalabilength}
\int_0^\infty \| n - S_{\o_{r_t}}\|_{L^p(X,({\o_{r_t}^n}/{\o^n})^q\o^n)}dt < \infty, 
\end{equation}
where $S_{\o_{r_t}}$ is the scalar curvature of $\o_{r_t}$. Additionally, if the above hold then $t \to r_t$ converges exponentially fast to a K\"ahler--Einstein metric.
\end{theorem}
As we shall see, the proof of this result requires no new a priori estimates, but instead rests on the observation that condition \eqref{eq: KRCalabilength} is equivalent to saying that the KR trajectory $t \to \o_{r_t}$ has finite length with respect to the $L^{p,q}$-Calabi Finsler metric on $\mathcal H$, that we introduce now. By Hodge theory, if $\o' \in \mathcal H$ then $\o' = \o_u := \o + i\ddbar u$ for some $u \in C^\infty(X)$, hence $\mathcal H$ can be identified with $\mathcal H_\o$, the set of normalized K\"ahler potentials\vspace{-0.1cm}:
$$\mathcal H_\o := \{u \in \mathcal C^\infty(X): \ \o_u = \o+ i\ddbar u >0, \ \int_X u \o^n =0 \}\simeq \mathcal H.\vspace{-0.1cm}$$
As our approach in this note makes heavy use of pluripotential theory, we will mostly work with potentials instead of metrics. Treating $\mathcal H_\o$ as a (trivial) Fr\`echet  manifold, one can introduce the $L^{p,q}$-Calabi Finsler metric for arbitrary $1 \leq q \leq p < \infty$:
\begin{equation}\label{eq: LpCalabiFinslerDef}
\|\beta\|^C_{p,q,u} = \bigg[\frac{1}{V}\int_X | \Delta_{\o_u} \beta |^p\Big[\frac{{\o_u}^n}{\o^n}\Big]^{{q}}\o^n\bigg]^{1/p}, \ \beta \in T_{u} \mathcal H_\o, \ u \in \mathcal H_\o,
\end{equation}
where $V=\int_X \o^n$ is the total volume. The case $p=2,q=1$ gives the Riemannian structure of Calabi \cite{clb1} recently studied extensively in \cite{clm,cr}. In the most important case $q=1$, we will simply refer to the metric in \eqref{eq: LpCalabiFinslerDef} as the $L^p$-Calabi metric.

Using this Finsler metric it is possible to compute the length of smooth curves, and introduce the associated path length pseudometric $d_{p,q}^C$ on $\mathcal H_\o$. It turns out that $(\mathcal H_\o,d_{p,q}^C)$ is a bona fide metric space and its completion and metric growth can be analytically characterized using elements from the finite energy pluripotential theory of Guedj--Zeriahi \cite{gz}, generalizing \cite[Theorem 5.4]{cr} in the process:
\begin{theorem}\label{thm: Calabicompletion} Suppose $(X^n,\o)$ is K\"ahler and $1 \leq q \leq p < \infty$. Then $(\mathcal H_\o,d_{p,q}^C)$ is a metric space whose completion is $(\mathcal E^{L^q},d_{p,q}^C)$. Characterizing convergence, a sequence $\{u_j\}_{j} \subset \mathcal H_\o$ is $d^C_{p,q}$-Cauchy if and only if \vspace{-0.1cm}
\begin{equation}\label{eq: Calabimetricomparison}
\int_X \bigg|\frac{\o_{u_j}^n}{\o^n}- \frac{\o_{u_k}^n}{\o^n} \bigg|^q \o^n \to 0 \textup{ as } j,k \to \infty.\vspace{-0.1cm}
\end{equation}
Additionally, $(\mathcal E^{L^1},d_{2,1}^C)$ is a $\textup{CAT({1}/{4})}$ geodesic metric space.
\end{theorem}
This theorem seems to give the first geometric characterization of the well known potential space $\mathcal E^{L^q}$. Roughly speaking, $\mathcal E^{L^q}$ contains degenerate metrics with volume measure having $L^q$-density, and we give now the precise definition, referring to \cite{gz} for additional details. In \cite{gz}, associated to any $u \in \textup{PSH}(X,\o)$ the authors introduce a non-pluripolar measure $\o_{u}^n$ on $X$, satisfying $\int_X \o_u^n \leq \int_X \o^n$, generalizing the usual complex Monge-Amp\`ere measure of Bedford--Taylor in case $u$ is additionally bounded. By definition, $u \in \mathcal E \subset \textup{PSH}(X,\o)$ if $u$ has ``full volume", i.e., $\int_X \o_u^n = \int_X \o^n$. Given $p,q\geq 1$, two important subclasses of potentials inside $\mathcal E$ are as follows: \vspace{-0.1cm} 
$$
\mathcal E^{L^q} := \big\{u \in \mathcal E, \int_X u \o^n =0, \ \frac{\o_{u}^n}{\o^n} \in L^q(X, \o^n)\big\}, \ \ \mathcal E^p:=\big\{u\in\mathcal E, \int_X u \o^n =0, \ \int|u|^p\o_u^n<\infty\big\}.
\vspace{-0.1cm}$$
\begin{remark} \label{rem: alld^M_peqv} By basic analysis, the $d_{p,q}^C$-convergence characterization \eqref{eq: Calabimetricomparison} extends to sequences inside the completion $\mathcal E^{L^q}$ as well. 

One can derive the equation for $L^{p,q}$-Calabi geodesics by computing the variation of the Finsler energy along curves with fixed endpoints. Contrary to the findings of \cite{clm} in the particular case $p=2,q=1$, this geodesic equation does not admit smooth solutions for general $p$ and $q$, and to avoid complications we omit the discussion of ``weak geodesics" in this note.

A distinguishing feature of the $L^{p,q}$-Calabi geometry is that the associated path length metric $d^C_{p,q}$ induces the same completion on $\mathcal H_\o$ for all $p \geq 1$ and fixed $q$, even though the corresponding Finsler metrics for different $p$ are not even fiberwise conformally equivalent. Indeed, as follows from \eqref{eq: Calabimetricomparison},  $d_{p,q}^C$-convergence does not depend on the value of $p$. We are not aware of other families of infinite dimensional Finsler structures that would enjoy this same property. See \cite{cr2} for a treatment of conformal deformations of the Ebin metric on the space of Riemannian metrics, where a related but different phenomenon occurs.

As we learned after the completion of this paper, in \cite[Section 4]{clzh}, motivated by different goals, a family of Riemannian metrics was introduced and studied in detail that overlaps with our construction of $L^{p,q}$-Calabi metrics when $p=2$.
\end{remark}
On top of applications to geometric flows, our motivation for studying the $L^{p,q}$ generalization of Calabi's Riemannian structure comes from the corresponding $L^p$ generalization of the Mabuchi geometry on $\mathcal H$ \cite{da2}, that led to many applications in the study of canonical K\"ahler metrics \cite{bbj,bdl,dh,dr}, and Theorem \ref{thm: Calabicompletion} stands in direct analogy with the findings of \cite{da1,da2} that we recall now. The $L^p$-Mabuchi Finsler metric of $\mathcal H_\o$ is defined as follows:
\begin{equation}\label{eq: LpFinsler}
\|\phi\|_{p,u} = \bigg[\frac{1}{V}{\int_X |\phi - \bar \phi_{\o_u}|^p \o_u^n}\bigg]^{1/p}, \  \phi \in T_u\mathcal H_\o, \  u \in \mathcal H_\o,
\end{equation}
where $\bar \phi_{\o_u}= \frac{1}{V}\int_X \phi \o_u^n$. 
The case $p=2$ gives the Riemannian structure of Mabuchi--Semmes--Donaldson \cite{ma,se,do}, a space with non-positive sectional curvature, with close ties to canonical K\"ahler metrics. As shown in \cite{dh,dr}, the case $p=1$ gives a geometry with good compactness properties, suitable for the variational study of canonical metrics by way of infinite-dimensional convex optimization. 

As in the case of the $L^{p,q}$-Calabi metric, using the Finsler structure of \eqref{eq: LpFinsler} we can measure the length of smooth curves, and for the associated path length pseudometric $d_p^M$ on $\mathcal H_\o$ we recall now the Mabuchi analog of Theorem \ref{thm: Calabicompletion}, concatenating the relevant parts of \cite[Theorem 1]{da1} and \cite[Theorem 2, Theorem 3]{da2}. We refer to \cite{da1,da2} for further details.
\begin{theorem}\label{thm: Mabuchicompletion} Suppose $(X^n,\o)$ is K\"ahler and $p \geq 1$. Then $(\mathcal H_\o,d^M_p)$ is a metric space whose completion is $(\mathcal E^p,d_p^M)$. Characterizing convergence, a sequence $\{ u_j\}_j \subset \mathcal H_\o$ is $d^M_{p}$-Cauchy if and only if \vspace{-0.05cm}
\begin{equation}\label{eq: Mabuchimetricomparison}
{\int_X | u_j - u_k |^p \o_{u_j}^n} + {\int_X | u_j - u_k |^p \o^n_{u_k}} \to 0 \textup{ as }j,k \to \infty.\vspace{-0.05cm}
\end{equation}
Additionally, $(\mathcal E^2_M,d_2^M)$ is a \textup{CAT(0)} geodesic metric space.
\end{theorem}

With Theorems \ref{thm: Calabicompletion} and \ref{thm: Mabuchicompletion} in hand, we can answer in a much more general context the questions of Clarke--Rubinstein \cite[Section 7.2]{cr}, who proposed to compare the $L^2$-Mabuchi and $L^2$-Calabi metric structures. It follows from the proof of Theorem \ref{thm: Calabicompletion}  that $(\mathcal H_\o,d^C_{p,1})$ has finite diameter. As $(\mathcal H_\o,d^M_{p'})$ has infinite diameter \cite{da2}, it is not possible for $d^C_{p,1}$ to globally dominate $d^M_{p'}$. Even in the absence of global metric domination, one may still hope that domination holds on the level of the induced topologies. In case $q > 1$, thanks to strong estimates of Kolodziej \cite[p. 668]{k}, the $d_{p,q}^C$-topology dominates the $C^0$-topology hence implicitly also the $d_{p'}^M$-topology. However the case $q=1$ is much more delicate and does not allow any kind of domination (see \cite[Theorem 1.3]{bl} for related a priori estimates), as we summarize in the following result that fully characterizes the relationship between the $L^{p,q}$-Calabi and $L^{p'}$-Mabuchi topologies:
\begin{theorem} \label{thm: MabuchiCalabinotcompatible} Suppose $(\mathcal H,\o)$ is K\"ahler, $1 \leq q \leq p <\infty$ and $1\leq p'$. The following hold:
\vspace{-0.05in}
\begin{itemize}
\setlength{\itemsep}{1pt}
    \setlength{\parskip}{1pt}
    \setlength{\parsep}{1pt} 
\item[(i)] There exists a sequence $\{ v_k\}_{k \in \Bbb N} \subset \mathcal H_\o$ that is $d^M_{p'}$-Cauchy but does not contain any $d_{p,q}^C$-Cauchy subsequences.
\item[(ii)] If $q > 1$, then any $d_{p,q}^C$-Cauchy sequence inside $\mathcal H_\o$ is $d_{p'}^M$-Cauchy as well. 
\item[(iii)] There exists a sequence $\{ u_k\}_{k \in \Bbb N} \subset \mathcal H_\o$ that is $d^C_{p,1}$-Cauchy but does not contain any $d^M_{p'}$-Cauchy subsequences.
\end{itemize} 
\end{theorem}

For the rest of the introduction, let us focus on the case $q=1$, most important from the point of view of geometric applications. An important step in the the proof of the previous theorem is noticing that the completion of $\mathcal H_\o$ with respect to the $L^p$-Calabi and $L^{p'}$-Mabuchi metrics cannot contain each other. Though containment is not possible, it is natural to search for interesting subspaces of the intersection $\overline{(\mathcal H_\o,d_{p'}^M)} \cap \overline{(\mathcal  H_\o,d_{p,1}^C)} = \mathcal E^{p'}\cap \mathcal E^{L^1}.$  

Given the importance of the $L^1$-Mabuchi metric in applications to existence/uniqueness of canonical K\"ahler metrics \cite{dr,dh,bbj}, we further restrict attention to the case $p\geq 1$ and $p'= 1$. A natural subspace of the intersection $\mathcal E^{1}\cap \mathcal E^{L^1}$ is the space of potentials with finite entropy, studied in \cite{bbegz,bdl} in connection with canonical K\"ahler metrics:
$$\mathcal E^\textup{Ent} := \{u \in \mathcal E: \textup{Ent}(\o^n,\o^n_u) < \infty, \  \int_X u \o^n=0\},$$
where $\textup{Ent}(\o^n,\o^n_u)=\infty$ if $\o^n_u$ is not absolutely continuous with respect to  $\o^n$, and is equal to $\int_X \log\big( \frac{\o^n_u}{\o^n}\big)\frac{\o^n_u}{\o^n} \o^n$  otherwise. By definition, $\mathcal E^\textup{Ent} \subset \mathcal E^{L^1}$ and it is well known that also $\mathcal E^\textup{Ent} \subset \mathcal E^1$ (see \cite{bbegz, bdl}), hence as proposed
$$\mathcal E^{\textup{Ent}} \subset \mathcal E^1 \cap \mathcal E^{L^1}.$$
It follows that $\mathcal E^{\textup{Ent}}$ can be endowed with two different non-complete topologies induced by $d^C_{p,1}$ and $d^M_1$. A natural way to make these topologies complete on $\mathcal E^{\textup{Ent}}$ is to strengthen them enough to make the map $u \to \textup{Ent}(\o^n,\o^n_u)$ continuous. It turns out this procedure gives equivalent topologies and in fact much more is true:
\begin{theorem} \label{thm: eqvonvergence} Suppose $u_j,u \in \mathcal E^{\textup{Ent}}$ satisfy $\textup{Ent}(\o^n,\o_{u_j}^n) \to \textup{Ent}(\o^n,\o_u^n) $. Then the following are equivalent:
\vspace{-0.1in}
\begin{itemize}
\setlength{\itemsep}{1pt}
    \setlength{\parskip}{1pt}
    \setlength{\parsep}{1pt}  
\item[(i)] ${u_j} \to {u}$ in $L^1(X,\o^n)$. 
\item[(ii)] $\o^n_{u_j} \to \o^n_{u}$ in the weak sense of measures.
\item[(iii)] $d^M_1(u_j,u) \to 0$.
\item[(iv)] $d^C_{p,1}(u_j,u) \to 0$.
\end{itemize}
\end{theorem}
An immediate application of the equivalence between (iii) and (iv) in the last theorem and \cite[Theorem 1.2, Theorem 1.11]{bdl} is a new convergence result for the weak Calabi flow. This weak flow is a generalization of the classical smooth Calabi flow \cite{clb2,cc} (that is governed by the equation $\frac{d}{dt}c_t = S_{\o_{c_t}}-\bar S$) and was initially introduced and studied by Streets in the context of the abstract metric completion $\overline{(\mathcal H,d_2^M)}$ \cite{st1,st2}. A better understanding of this latter space in \cite{da1} led to more precise long time convergence and asymptotics results for the weak Calabi flow in \cite{bdl} and for more details, related terminology and historic references we refer to this last paper. We note the following corollary, which is a direct consequence of Theorem \ref{thm: eqvonvergence} above and \cite[Theorem 1.5, Theorem 1.11]{bdl}:

\begin{corollary} Suppose $p \geq 1$ and there exists a constant scalar curvature metric in $\mathcal H$. Then, given  any weak Calabi flow trajectory $[0,\infty) \ni t \to c_t \in \mathcal E^2$, there exists a constant scalar curvature potential $c_\infty \in \mathcal H_\o$ such that $d_{p,1}^C(c_t,c_\infty) \to 0$, in particular $\int_X \big|\frac{\o_{c_t}^n}{\o^n} - \frac{\o_{c_\infty}^n}{\o^n}\big| \o^n \to 0$.
\end{corollary}

\section{The $L^{p,q}$-Calabi geometry of $\mathcal H_\o$}

As we will see, $\mathcal H_\o$ equipped with the $L^{p,q}$-Calabi Finsler structure can be embedded isometrically into $L^p(X,\o^n)$. For this reason, we focus on the ``flat" geometry of $L^p(X,\o^n)$ in the next short subsection.

\subsection{$L^p$-geometry on $L^{p/q}$-spheres}

Let $1 \leq q \leq p < \infty$ and $X$ be a compact manifold with a positive Borel measure $\mu$ satisfying $\mu(X)=\int_X \mu < \infty$. As a Fr\'echet manifold, $C^\infty(X)$ can be equipped with the  trivial $L^p$-Finsler structure:
\begin{equation}\label{eq: trivLpmetric}
\| \psi \|_{p,f} = \Big(\frac{1}{\mu(X)}\int_X |\psi|^p\mu\Big)^{1/p}, \ \psi \in T_f{C^\infty(X)} \simeq C^\infty(X), \ f \in C^\infty(X).
\end{equation}
It is a classical fact that straight segments joining various points of $C^\infty(X)$ are geodesics for this metric. The $L^{p/q}$-sphere with radius $r$ is denoted by:
$$\Bbb S_{L^{p/q}}(\mu,r) = \{f \in C^\infty(X) \textup{ s.t. } \frac{1}{\mu(X)}\int_X |f|^{p/q} \mu =r^{p/q}\}.$$
Unfortunately, for most $p$ and $q$, the $L^{p/q}$-sphere is not even a smooth submanifold of $C^\infty(X)$, however if we restrict to the ``octant" $\Bbb S_{L^{p/q}}^+(\mu,r)=\Bbb S_{L^{p/q}}(\mu,r) \cap \{ f > 0\}$, we do get a smooth submanifold.  As such, one can pullback the $L^p$-Finsler metric of \eqref{eq: trivLpmetric} to $\Bbb S_{L^{p/q}}^+(\mu,r)$, and study the resulting path length metric space $(\Bbb S_{L^{p/q}}^+(\mu,r), d_{p,q}^{\Bbb S^+})$. We will need the following basic result in this direction, which roughly says that the ``chordal metric" is equivalent to the ``round metric" on $\Bbb S^+_{L^{p/q}}(\mu,r)$:
\begin{proposition} Fix $f \in \Bbb S_{L^{p/q}}^+(\mu,r)$. Then there exists $C:=C(\mu,p,q,f,r) > 0$ such that for any $f_0,f_1 \in \Bbb S^{+}_{L^{p/q}}(\mu,r)$ the following holds:
\begin{equation}\label{eq: cordalsphericalmetriccomp}
\frac{C}{d^{\Bbb S^+}_{p,q}(f,f_0)+d^{\Bbb S^+}_{p,q}(f,f_1) + 1} d_{p,q}^{\Bbb S^+}(f_0,f_1) \leq \Big(\frac{1}{\mu(X)}\int_X |f_0-f_1|^p\mu\Big)^{1/p}\leq d_{p,q}^{\Bbb S^+}(f_0,f_1).
\end{equation}
\end{proposition}
\begin{proof} The inequality $(\frac{1}{\mu(X)}\int_X |f_0-f_1|^p\mu)^{1/p}\leq  d_{p,q}^{\Bbb S^+}(f_0,f_1) $ follows from the fact that $[0,1] \ni t \to f_t := f_0 + t(f_1-f_0) \in C^\infty(X)$ is a geodesic in $C^\infty(X)$  with respect to the $L^p$-Finsler metric and has length equal to $(\frac{1}{\mu(X)}\int_X |f_0-f_1|^p\mu)^{1/p}$. Any curve joining $f_0,f_1$ inside $\Bbb S^{+}_{L^{p/q}}(\mu,r)$ has to have length less than $t \to f_t$.  

For the other inequality, we will estimate the length of the curve
$$[0,1] \ni t \to \alpha_t := \frac{rf_t}{\| f_t\|_{p/q}}=\frac{r(f_0 + t(f_1 - f_0))}{\| f_0 + t(f_1 - f_0)\|_{p/q}} \in \Bbb S_{L^{p/q}}^+(\mu,r),$$ 
joining $f_0,f_1$. Note that the denominator of the expression above is nonzero, as $f_0,f_1>0$. Using that 
$$\dot \alpha_t = \frac{r(f_1 - f_0)}{\| f_t\|_{p/q}} -\frac{ rf_t }{\| f_t\|_{p/q}^{p/q+1}} \int_X (f_1 - f_0) f_t^{p/q-1} d\mu,$$ 
we have the following sequence of estimates:
\begin{flalign*}
\int_0^1 \| \dot \alpha_t\|_pdt &\leq \int_0^1\frac{r\|f_1 - f_0 \|_{p}}{\| f_t\|_{p/q}}dt + \int_0^1 \frac{r\| f_t\|_{p}}{\| f_t\|^{p/q+1}_{p/q}}\int_X |f_1 - f_0|f_t^{p/q-1}d\mu dt\\
&\leq \int_0^1\frac{r\|f_1 - f_0 \|_{p}}{\| f_t\|_{p/q}}dt + \int_0^1\frac{r\| f_t\|_{p}}{\| f_t\|^{p/q+1}_{p/q}} \|f_1 - f_0 \|_{p/q} \|f_t\|_{p/q}^{p/q-1}dt \\
&\leq r\|f_1 - f_0 \|_{p} \int_0^1\frac{1}{\| f_t\|_{p/q}}dt + C'(\mu,p,q)\|f_1 - f_0 \|_{p} \int_0^1\frac{\| f_t\|_p}{\| f_t\|_{p/q}^2}dt\\
&\leq r\|f_1 - f_0 \|_{p} \int_0^1\frac{1}{\| f_t\|_{p/q}} + C'(\mu,p,q)\frac{(1-t)\| f-f_0\|_p + t\| f -f_1 \|_p + \| f\|_p}{\| f_t\|_{p/q}^2}dt\\
&\leq C(\mu,p,q,r,f)(d^{\Bbb S^+}_{p,q}(f,f_0)+d^{\Bbb S^+}_{p,q}(f,f_1) + 1) \|f_1 - f_0 \|_{p},
\end{flalign*}
where to obtain the second line we have used the H\"older inequality with exponents $p/q\geq 1$ and $(p/q)/(p/q-1) \geq 1$ in the last integrand. To get the third line, we have used that $\| f_0 -f_1\|_{p/q} \leq C'(\mu,p,q) \|f_0 - f_1 \|_{p}$. For the fourth line, we have used the triangle inequality for the $L^p$-norm. To get the last line, we have used that $\|f -f_0 \|_{p} \leq d_{p,q}^{\Bbb S^+}(f,f_0), \ \|f -f_1 \|_{p} \leq d_{p,q}^{\Bbb S^+}(f,f_1)$, $f_t \geq f_0/2 \geq 0$ for $t \in [0,1/2]$ and $f_t \geq f_1/2 \geq 0$ for $t \in [1/2,1]$, hence $\| f_t\|_{p/q} \geq r/2$ for all $t \in [0,1]$.  To finish the proof, we conclude:
$$d_{p,q}^{\Bbb S^+}(f_0,f_1) \leq \int_0^1 \| \dot \alpha_t\|_p dt \leq C(\mu,p,q,f)(d^{\Bbb S^+}_{p,q}(f,f_0)+d^{\Bbb S^+}_{p,q}(f,f_1) + 1) \|f_1 - f_0 \|_{p}.$$
\end{proof}

\subsection{Proof of Theorem \ref{thm: Calabicompletion} and Theorem \ref{thm: KRconvergence}}

\begin{proof}[Proof of Theorem \ref{thm: Calabicompletion}] 
To start the proof, we notice that for arbitrary $1 \leq q \leq p < \infty$ the infinite-dimensional Finsler manifolds $(\mathcal H_\o,\| \cdot \|^C_{p,q,(\cdot)})$ and $(\Bbb S^+_{L^{p/q}}(\o^n,p/q),\| \cdot \|_{p,(\cdot)})$ are isometric via the map $F: \mathcal H_\o \to \Bbb S^+_{L^{p/q}}(\o^n,p/q)$, given by the formula
$$F(u) := \frac{p}{q}\Big(\frac{\o_u^n}{\o^n}\Big)^{\frac{q}{p}}.$$
By the Calabi--Yau theorem, the map $F$ is bijective. As $F(u)_*(\delta u) = ( {\o_u^n}/{\o^n})^{q/p}\Delta_{\o_u}\delta u$, by inspection we see that $F^*\| \cdot \|_{p,F(\cdot)}=\| \cdot \|^C_{p,q,(\cdot)}$. All this implies that
$$d_{p,q}^C(u_0,u_1) = d_{p,q}^{\Bbb S^+}(F(u_0),F(u_1)),$$
in particular, \eqref{eq: cordalsphericalmetriccomp} gives that $d_{p,q}^C$ is indeed a metric on $\mathcal H_\o$. From \eqref{eq: cordalsphericalmetriccomp} it also follows that $\{ u_j\}_j \subset \mathcal H_\o$ is $d_{p,q}^C$-Cauchy if and only if
$$\int_X \Big|\Big(\frac{\o_{u_j}^n}{\o^n}\Big)^{q/p} - \Big(\frac{\o_{u_k}^n}{\o^n}\Big)^{q/p}\Big|^p \o^n \to 0, \ j,k \to \infty.$$ 
Using this, Lemma \ref{lemma: Vitali} below implies that $\{ u_j\}_j$ is $d_{p,q}^C$-Cauchy if and only if \eqref{eq: Calabimetricomparison} holds. The identification $\overline{(\mathcal H_\o,d_{p,q}^C)}=\mathcal E^{L^q}$ readily follows as well.

Lastly, we focus on the case $p=2,q=1$ extensively treated in \cite{clm,cr}. As observed in \cite[Theorem 1.1]{clm} (see also the discussion following \cite[Remark 4.2]{cr}), the Riemannian space $(\mathcal H_\o,\| \cdot \|_{2,1,(\cdot)}^C)$ has positive constant sectional curvature equal to $1/4$, what is more, any two points of $\mathcal H_\o$ can be joined by a Riemannian geodesic. Roughly, this follows from the fact that $(\mathcal H_\o,\| \cdot \|^C_{2,1,(\cdot)})$ is isometric to $(\Bbb S^+_{L^2}(\o^n,2),\| \cdot \|_{2,(\cdot)})$, which is a totally geodesic open subset of an infinite-dimensional sphere with radius $2$. 

Given $u,v,w \in \mathcal H_\o$, let $U=F(u),V=F(v),W=F(w) \in \Bbb S^+_{L^2}(\o^n)$. Also let $\mathcal V \subset L^2(X,\o^n)$ be the 3 dimensional subspace spanned by $U,V,W$ and $\Bbb S_{UVW}= \mathcal V \cap \Bbb S^+_{L^2}(\o^n,2)$. Together with the induced Riemannian metric, $\Bbb S_{UVW}$ is isometric to an open subset of the 2-dimensional round sphere with its round metric, hence it has constant sectional curvature $1/4$. The geodesic triangle $UVW$ of $\Bbb S^+_{L^2}(\o^n,2)$, with edges at $U,V,W$ lies inside $\Bbb S_{UVW}$ \cite[Theorem 1.4]{clm}. As $\Bbb S_{UVW}$ is a model space with constant scalar curvature equal to $1/4$, the geodesic triangle $UVW$ (lying inside it) has to satisfy the CAT(1/4) inequality \cite{bh}, ultimately giving that $(\mathcal H_\o,d_2^C)$ is a CAT(1/4) space. 

To finish the proof, we can use \cite[Corollary 3.11, p. 187]{bh} to conclude that the metric completion $\overline{(\mathcal H_\o,d_2^C)}=(\mathcal E^{L^1},d_2^C)$ is a CAT(1/4) geodesic metric space as well.
\end{proof}

As promised in the above proof, let us state the following measure theoretic lemma, whose proof uses the classical Vitali convergence theorem \cite[Theorem 8.5.14]{ra}, and is exactly the same as the argument of \cite[Lemma 5.3]{cr}: 
\begin{lemma} \label{lemma: Vitali} Suppose $f_j,f \in L^q(X,\o^n)$ with $f_j,f \geq 0$. Then $\|f_j -f \|_{L^q} \to 0$ if and only if $\| f_j^{q/p} -{f}^{q/p}\|_{L^p}\to 0$.
\end{lemma}

Lastly, we provide the following theorem, which contains Theorem \ref{thm: KRconvergence} as a particular case. We refer to \cite{da2,dh,mc} for analogous results on the $L^{P'}$-Mabuchi convergence of the K\"ahler--Ricci flow.

\begin{theorem}\label{thm: KRconvergence_general}Suppose $(X^n,\o)$ is Fano with $[\o]=-c_1(K_X)$ and $1\leq q \leq p \leq \infty$. Suppose $[0,\infty) \ni t \to r_t \in \mathcal H_\o$ is a K\"ahler-Ricci trajectory. Then the following are equivalent:
\vspace{-0.1in}
\begin{itemize}
\setlength{\itemsep}{1pt}
    \setlength{\parskip}{1pt}
    \setlength{\parsep}{1pt}
\item[(i)] There exists a K\"ahler--Einstein potential inside $\mathcal H_\o$.  
\item[(ii)] $t \to r_t$ converges $C^\infty$-exponentially fast to some K\"ahler--Einstein potential $r_\infty \in \mathcal H_\o$.
\item[(iii)] $t \to r_t$ has finite $d^C_{p,1}$-length, i.e., $\int_0^\infty \| n - S_{\o_{r_t}}\|_{L^p(X,({\o_{r_t}^n}/{\o^n})^q\o^n)} < \infty.$
\item[(iv)] $\{r_t\}_{t \geq 0} \subset \mathcal H_\o$ forms a $d^C_{p,q}$-Cauchy sequence.
\end{itemize}
\end{theorem}
\begin{proof}
Let us first assume that $p \neq \infty$. If (i) holds then by results of Perelman, Tian--Zhu, Phong--Song--Sturm--Weinkove and Collins-Sz\'ekelyhidi \cite{tz,pssw,csz} imply that the K\"ahler--Ricci trajectory $t \to r_t$ converges $C^\infty$ exponentially fast to some K\"ahler--Einstein potential $r_\infty \in \mathcal H_\o$, hence (ii) holds. $C^\infty$-exponential convergence of $t \to r_t$ implies the finiteness of the integral in (iii). Condition (iii) implies (iv) trivially.

We are left to show that (iv) implies (i). The the ideas of  \cite[Corollary 6.7]{cr} apply again, but we give here a slightly different argument. Suppose (iv) holds but (i) does not. Let $r_\infty \in \mathcal E^{L^q}$ be the $d_{p,q}^C$-limit of $r_t$. It follows from the convergence criterion of \eqref{eq: Calabimetricomparison} and Remark \ref{rem: Calabiconvimpliespointwiseconv} below that we also have weak convergence on the level of potentials, namely $r_t \to_{L^1(X,\o^n)} r_\infty$. On the other hand, by \cite[Theorem 1.3]{r}, there exists $t_j \to \infty$ and $\psi \in \textup{PSH}(X,\o)$ such that $\psi$ has proper multiplier ideal sheaf, in particular by Skoda's theorem $\psi$ has non-zero Lelong number at some $x \in X$. By \cite[Corollary 1.8]{gz} this implies that $\psi \not \in \mathcal E$. But by uniqueness of $L^1(X,\o^n)$-limits, we have $\psi=r_\infty \in \mathcal E^{L^q} \subset \mathcal E$, a contradiction. 

Now we deal with the case $p=\infty$. Clearly, the directions (i)$\to$(ii)$\to$(iii)$\to$(iv) still hold. To prove that (iv) implies (i) we just need to notice that $d_{\infty}^C$-convergence trivially implies $d_{p,q}^C$-convergence for any $1 \leq q \leq p < \infty$.
\end{proof}
\section{$L^{p,q}$-Calabi vs. $L^{p'}$-Mabuchi  geometry}
\subsection{Proof of Theorem \ref{thm: MabuchiCalabinotcompatible}}
To prove (i), we recall first that $\mathcal E^{p'} \not \subset \mathcal E^{L^q}$. Indeed, we can choose $v_0,v_1 \in \mathcal H_\o$ such that the level set $\{ v_0-v_1=0\}$ does not contain critical points of $v_0-v_1$. Then a basic calculation yields that the bounded potential  $u=\max(v_0,v_1)-\int_X\max(v_0,v_1)\o^n$ satisfies $u \in \mathcal E^{p'} \setminus \mathcal E^{L^q}$, because $\o^n_u$ charges the hypersurface $\{ v_0=v_1\}$, a set of Lebesgue measure zero. 

Now let $u \in \mathcal E^{p'} \setminus \mathcal E^{L^q}$ arbitrary. By Theorem \ref{thm: Mabuchicompletion} there exists $u_j \in \mathcal H_\o$ such that $d_{p'}^M(u_j,u) \to 0$. This in particular gives that $\o^n_{u_j} \to \o^n_u$ weakly \cite[Theorem 5(i)]{da2}. We claim that $\{ u_j\}_j$ cannot contain a $d_{p,q}^C$-Cauchy subsequence. Indeed, if this were the case, then by Theorem \ref{thm: Calabicompletion} above, for some subsequence of $u_j$, again denoted by $u_j$, the densities $\o_{u_j}^n/\o^n$ would converge in $L^q(X,\o^n)$ to some $f \in L^q(X,\o)$. But as $\o^n_{u_j} \to \o^n_u$ weakly, this would imply that $\o^n_u/\o^n = f \in L^q(X,\o^n)$, a contradiction with $u \in \mathcal E^{p'} \setminus \mathcal E^{L^q}$. 

To argue (iii), we first show that $ \mathcal E^{L^1} \not \subset \mathcal E^{p'} $. This is again likely known to experts, however we could find an exact reference, so we give a construction allowing a great amount of flexibility. Let $u \in \mathcal E^{p'}$, $u \leq -1$ and unbounded such that for each set $U_k = \{ k < |u|^{p'} \leq k+1\}$ we have $\o^n(U_k)>0, \ k\geq 1$. By the construction in \cite[Example 2.14]{gz} (see also \cite[Proposition 5]{bl2}), such $u$ can be found. We introduce $f \in L^1(X,\o^n)$:
$$f(x) = \sum_{k \geq 1} \frac{6V}{(\pi k)^2 \o^n(U_k)}\mathbbm{1}_{U_k}(x).$$
Clearly $f \in L^1(X)$ with $\int_X f \o^n =V$, hence by \cite[Theorem A]{gz} there exists $v \in \mathcal E^{L^1}$ such that $\o_v^n = f \o^n$. We claim that $v \notin \mathcal E^{p'}.$ Indeed, if this were not true, then \cite[Theorem C]{gz} would give that
$$\infty=\frac{6V}{\pi^2}\sum_{k \geq 1} \frac{1}{k} \leq \int_X |u|^{p'} \o_v^n < \infty,$$
a contradiction. Finally, as $v\in \mathcal E^{L^1}\setminus \mathcal E^{p'},$ the same argument as in the previous step  yields now a $d^C_{p,1}$-Cauchy sequence $\{v_j\}_j \subset \mathcal H_\o$ for which $d^C_{p,1}(v_j,v) \to 0$, without any $d^M_{p'}$-Cauchy subsequences.

Finally, to argue (ii), we have to use jointly the $d^C_{p,q}$-convergence criteria \eqref{eq: Calabimetricomparison} and the estimates of Kolodziej \cite[p. 668]{k}, according to which $d^C_{p,q}$-convergence implies $C^0$-convergence of potentials. According to the  $d^M_{p'}$-convergence criteria \eqref{eq: Mabuchimetricomparison}, $C^0$-convergence in turn implies $d^M_{p'}$-convergence, finishing the proof.
 
\subsection{Proof of Theorem \ref{thm: eqvonvergence}}

The following basic consequence of the dominated convergence theorem will be used shortly:

\begin{lemma} Suppose $f \geq 0$ such that $\int_X f \log(f) \o^n < \infty$. Then there exists $\tilde f_k \in C^\infty(X)$ such that $\tilde f_k >0$, $\int_X |f - \tilde f_k|\o^n \to 0$ and $\int_X f(\log(f) - \log( \tilde f_k))\o^n \to 0$.
\end{lemma}

\begin{proof} Let $f^m = \max\{\min\{f,m \},1/m\}$. By the dominated convergence theorem every sequence $\{g_j\}_j \subset C^\infty(X)$ satisfying $m> g_j > 1/m$ and $\int_X |f^m -g_j|\o^n \to 0$ contains an element $\tilde f_m := g_{j_m}$ with $\int_X f (\log(f^m) - \log (\tilde f_m))\o^n \leq 1/n$ and $\int_X |f^m -\tilde f_m|\o^n \leq 1/n$. By the absolute continuity of the Lebesgue integral, it follows that $\{ \tilde f_k\}_k$ satisfies the properties of the lemma.
\end{proof}

\begin{proof}[Proof of Theorem \ref{thm: eqvonvergence}] First we show the equivalence between (i),(ii) and (iii). From \cite[Theorem 5(i)]{da2} it follows that (iii)$\to$(i) and (iii)$\to$(ii).

The proof of (i)$\to$(iii) and (ii)$\to$(iii) are almost the same and we only carry out the latter. It follows from the compactness theorem \cite[Theorem 2.17]{bbegz} (for a statement most suitable for our purposes see \cite[Theorem 5.6]{dr}) that any subsequence of $\{ u_j\}$ contains a subsubsequence $\{ u_{j_k}\}$ such that $d_1(u_{j_k},v) \to 0$ for some $v \in \mathcal E^{\textup{Ent}}$. If we can show that $v=u$ then we are done. By  \cite[Theorem 5(i)]{da2} again, we have $\o^n_{u_j} \to \o_v^n$ weakly, hence by the assumption we get $\o^n_u = \o^n_v$. As $u,v \in \mathcal E^1$, by uniqueness \cite[Theorem B]{gz} (see \cite{di} for a more general result), we conclude that $v=u$.

The direction (iv)$\to$(ii) is trivial and the main step is to argue that (ii)$\to$(iv). Let $f_j = \o^n_{u_j}/\o^n$ and $f = \o^n_{u}/\o^n$. By the $d^C_{p,1}$-convergence criteria \eqref{eq: Calabimetricomparison}, we have to show that $\int_X |f -f_j|\o^n \to 0$. Let $\tilde f_j \in C^\infty(X)$ be the sequence from the previous lemma. For fixed $k$ we have
$$ \lim_{j}\int_X |f -f_j|\o^n \leq 
\limsup_j\int_X |\tilde f_k -f_j|\o^n + \int_X |f -\tilde f_k|\o^n,$$
hence we only need to check that the first term on the right hand side goes to zero as $k \to \infty$. Using the classical K\"ullback-Pinsker inequality (see \cite[Proposition 2.10(ii)]{bbegz} for statement tailored to our setting) we have the following sequence of estimates:
\begin{flalign}\label{calc1}
\nonumber \limsup_j \bigg(\int_X |f_j - \tilde f_k| \o^n\bigg)^2 &\leq \limsup_j\int f_j \log\Big(\frac{f_j}{\tilde f_k}\Big)\o^n\\
\nonumber & \leq \limsup_j \int f_j \log {f_j}\o^n -\liminf_j\int f_j \log \tilde f_k\o^n \\
& = \int f \log {f}\o^n -\int f \log \tilde f_k\o^n, 
\end{flalign}
where to get the last line we have used that $\textup{Ent}(\o_{u_j},\o) \to \textup{Ent}(\o_u, \o)$ and that $\o^n_{u_j}=f_j \o^n$ converges weakly to $\o^n_u$. By the previous lemma, the expression in \eqref{calc1} tends to zero as $k \to \infty$, hence we are done.
\end{proof}

It is perhaps worth noting that (iv) implies (ii) in Theorem \ref{thm: eqvonvergence} without the assumption on the convergence of entropy, as we elaborate now. Suppose $u_j,u \in \mathcal E^{L^1}$ with $d^M_{p,1}(u_j,u)\to 0$. As $\{ u_j\}_j$ is $L^1(X,\o^n)$-compact (since $\int_X u_j \o^n = 0$), we have to argue that any $L^1(X,\o)$-convergent subsequence of $\{ u_j\}$ $L^1$-converges to $u$. Suppose $u_{j_k} \to_{L^1} v \in \textup{PSH}(X,\o)$. By \cite[Proposition 2.10(i)]{begz} it follows that $v \in \mathcal E^1$, in particular $\o^n_v$ has full mass ($u \in \mathcal E$). As a consequence of  \cite[Corollary 2.21]{begz} we now obtain that $\o^n_v \geq \o^n_u$. As both of these last measures have the same total volume we have in fact $\o^n_v = \o^n_u$, hence  $v=u$ as desired (here we used again \cite[Theorem B]{gz}). 

For $q>1$,  $d^C_{p,q}$-convergence implies $C^0$-convergence (hence also $L^1(X,\o^n)$-convergence of potentials) as was noted in the proof of Theorem \ref{thm: MabuchiCalabinotcompatible}, and we summarize our findings in the next remark, obtaining a partial analog of \cite[Theorem 5(i)]{da2} for the $d^C_{p,q}$ metric in the process:

\begin{remark}\label{rem: Calabiconvimpliespointwiseconv} Suppose $1 \leq q \leq p<\infty$ and $u_j,u \in \mathcal E^{L^q}$. Then $d^C_{p,q}(u_j,u) \to 0$ implies that $u_j \to_{L^1(X,\o^n)} u$.
\end{remark}

\paragraph{Acknowledgements.}
This research was supported by BSF grant 2012236.

\bigskip

{\sc University of Maryland} 

{\tt tdarvas@math.umd.edu}
\end{document}